\documentclass[12pt]{amsart}
\usepackage{latexsym}
\usepackage{amsmath}
\usepackage{amsfonts}
\usepackage{amsthm}
\usepackage{amssymb}
\usepackage{ifthen}
\usepackage{enumerate}
\usepackage[T1]{fontenc}
\def\Eq#1#2{\ifthenelse{\equal{#1}{*}}
  {\begin{equation*}\begin{aligned}#2\end{aligned}\end{equation*}}
  {\begin{equation}\begin{aligned}\label{E#1}#2\end{aligned}\end{equation}}}

\newtheorem{thm}{Theorem}
\newtheorem{prop}{Proposition}
\newtheorem{coro}{Corollary}
\newtheorem{lemm}{Lemma}

\theoremstyle{remark}

\newtheorem{claim}{Claim} 
\newtheorem{prclm}{Proof of Claim}

\theoremstyle{definition}

\newcommand{\dlth}{\mathcal{H}}
\newcommand{\dlta}{\mathcal{A}}
\newcommand{\dltb}{\mathcal{B}}

\newcommand{\NN}{\mathbb{N}} 

\newcommand{\QQ}{\mathbb{Q}}
\newcommand{\RR}{\mathbb{R}}

\newcommand{\cnc}{\mathfrak{c}}

\newcommand{\RRx}{\overline{\mathbb{R}}}

\newcommand{\card}{\mbox{\rm card}}
\newcommand{\supp}{\mbox{\rm supp}}

\newcommand{\map}{\longrightarrow}
\newcommand{\rfl}[2]{\left( #1 \, \vert \, #2 \right)}
 
\author[P\'eter T\'oth]{P\'eter T\'oth}

\title[Measurable solutions of an alternative functional equation]{Measurable solutions of an alternative functional equation} 

\address{Institute of Mathematics,
University of Debrecen,
4002 Debrecen, Pf.~400, Hungary}
\email{toth.peter@science.unideb.hu}

\keywords{functional equation, alternative functional equation, 
Darboux functions, measurable functions, 
generalized weighted quasi-arithmetic means} 

\subjclass[2020]{39B22, 26A15}

\thanks{The research has been supported by 
the EK\"OP-24-0 University Research Scholarship Program of 
the Ministry for Culture and Innovation 
from the source of the 
National Research, Development and Innovation Fund, 
and by the PhD Excellence Scholarship from the 
Count Istv\'an Tisza Foundation 
for the University of Debrecen.}

\begin{document}

\begin{abstract}
In this paper we investigate the functional equation 
\[ 
\varphi \left( \frac{x+y}{2} \right) 
\left( \psi_1(x) - \psi_2(y) \right) = 0 
\hspace{20mm} 
\left( \mbox{ for all } 
x \in I_1 \mbox{ and } y \in I_2 \right) 
\] 
where $ I_1 \,, I_2 $ are open intervals of $ \RR $, 
$ J = \frac{1}{2} \left( I_1 + I_2 \right) $ moreover 
$ \psi_1 : I_1 \rightarrow \RR $, 
$ \psi_2 : I_2 \rightarrow \RR $ and 
$ \varphi : J \rightarrow \RR $ 
are unknown functions. 
We describe the structure of the possible solutions 
assuming that $ \varphi $ is measurable. 
In the case when $ \varphi $ is a derivative, we 
give a complete characterization of the solutions. 
Furthermore, we present an example of a solution 
consisting of irregular Darboux functions. 
This provides the answer to an open problem 
proposed during the 
{\em $59$th International Symposium on Functional Equations}. 
\end{abstract}

\maketitle

\section{Introduction} 
During the whole paper let 
$ I_1 \,, I_2 $ be 
nonempty open intervals of the real line, and let 
$ J := \frac{1}{2} \left( I_1 + I_2 \right) $. 
Then $ J $ is again a nonempty, open interval. 
We will consider the functional equation 
\begin{equation}\label{eq-maineq} 
\varphi \left( \frac{x+y}{2} \right) 
\bigl( \psi_1(x) - \psi_2(y) \bigr) = 0 
\hspace{20mm} 
\left( \mbox{ for all } 
x \in I_1 \mbox{ and } y \in I_2 \right). 
\end{equation} 
If there exist functions 
$ \psi_1 : I_1 \map \RR $, 
$ \psi_2 : I_2 \map \RR $ and 
$ \varphi : J \map \RR $ 
such that the equation \eqref{eq-maineq}
is fulfilled for all 
$ x \in I_1 $ and for all $ y \in I_2 \, $, then 
we say that the triplet of functions 
$ \left( \varphi \,, \psi_1 \,, \psi_2 \right) $ 
is a solution of \eqref{eq-maineq}. 
This functional equation appears during the investigation 
of the invariance problem of 
generalized weighted quasi-arithmetic means. 
Since this class of means was introduced by 
J.~Matkowski in 2010 \cite{Mat10}, the 
generalized weighted quasi-arithmetic means are often referred to 
as Matkowski means. 
Special cases of the general invariance problem 
for Matkowski means have been investigated by many 
authors, such as Matkowski \cite{Mat99}, 
Jarczyk \cite{JM06, Jar07}, 
Daróczy, Páles \cite{DP01, DP02}, 
Maksa \cite{DMP00} and 
Burai \cite{Bur06, Bur07}. 
For a summary about invariance problems 
of means the interested reader 
should study the survey paper of 
Jarczyk and Jarczyk \cite{JJ18}. 

The general equivalence problem of Matkowski means 
can be formulated as follows. 
If $ \emptyset \neq J \subseteq \RR $ is an open interval, 
then we are looking for those Matkowski means 
$ M, N, K : J \times J \map J $ 
which fulfill the invariance equation 
\begin{equation}\label{eq-invariance_eq}
M \left( N(u,v) , K(u,v) \right) = M(u,v) 
\hspace{20mm} 
(u,v \in J). 
\end{equation} 
A recent achievement in the investigation of 
the general invariance problem of Matkowski means is 
the work of T.~Kiss \cite{Kis22}. In that paper 
-- based on certain ideas of Z.~ Páles -- 
a particular functional equation 
is established which is equivalent to 
\eqref{eq-invariance_eq}. 
This equivalent equation is 
\begin{equation}\label{eq-inveq_equiv}
F \left( \frac{x+y}{2} \right) + f_1(x) + f_2(y) = 
G \left( g_1(x) + g_2(y) \right) 
\end{equation}
where $ x,y $ are from an open interval. 

Our considered equation \eqref{eq-maineq} can be 
derived from \eqref{eq-inveq_equiv}, such that 
$ \varphi , \psi_1 \,, \psi_2 $ are 
obtained as compositions of functions 
which either appear in \eqref{eq-inveq_equiv} 
or they are derivatives of the functions in \eqref{eq-inveq_equiv}. 
However, functions themselves in \eqref{eq-inveq_equiv} 
are put together from the generator functions of 
the Matkowski means in the invariance problem. 
Thus it is clear that the initial regularity assumption for 
the generators determines the level of regularity 
we can assume when working on our main equation 
\eqref{eq-maineq}. 

The solutions of equation \eqref{eq-maineq} 
were characterized by Kiss \cite{Kis24} under the 
assumption that the set of zeros of 
$ \varphi $ is closed. 
The primary information about the solutions is that 
these must be constant on some open subintervals of 
their domain. This implies some consequences 
which are crucial when one wishes to determine 
the solutions of the equation \eqref{eq-inveq_equiv} 
and thus the invariance equation \eqref{eq-invariance_eq}. 
On the other hand, when 
$ \varphi $ is obtained from \eqref{eq-inveq_equiv} 
with the previously mentioned process, 
then $ \varphi $ is a derivative, 
so $ \varphi^{-1}(0) $ is not necessarily closed.
This is the reason why 
during the Problems and Remarks session of the 
59th International Symposium on Functional Equations 
T.~Kiss proposed 
the following question (see \cite{59ISFE}). 
Does the characterization \cite[Theorem 6]{Kis24} 
of the solutions of \eqref{eq-maineq} remain valid when the 
Darboux property is assumed for $ \varphi $, 
instead of the closedness of $ \varphi^{-1} (0) $? 

This would mean that \eqref{eq-inveq_equiv} could 
be solved imposing only differentiabilty on the 
functions in the equation, 
without requiring the continuity of the derivatives. 
To our best knowledge, 
continuous differentiability is 
the weakest regularity assumption under which 
the solutions of \eqref{eq-inveq_equiv} are discussed. 
However, it seems reasonable to investigate 
\eqref{eq-inveq_equiv} and thus the invariance problem of 
Matkowski means under weaker, more natural 
regularity assumptions. 
The main result of our paper will make it possible 
to consider \eqref{eq-inveq_equiv} 
assuming that the functions in the equation are 
differentiable (but not necessarily continuously 
differentiable). 

Let us mention that a similar process of 
weakening the regularity assumptions occurred during 
the investigation of the Matkowksi--Sut\^o problem. 
In that case, firstly 
the analytic solutions have been described by 
Sut\^o \cite{Sut14a, Sut14b}, 
then the assumption was weakened to 
twice continuous differentiability 
by Matkowski \cite{Mat99}. 
Dar\'oczy and P\'ales solved the Matkowski--Sut\^o 
equation under continuous differentiability 
\cite{DP01}, and finally they were able to 
solve it assuming only that the functions 
are continuous, strictly monotone (see \cite{DP02}). 

In our paper we are going to give a negative answer 
to the proposed problem of Kiss. 
We give a counterexample, namely a solution of 
\eqref{eq-maineq} consisting of Darboux functions 
which are not constant on any open subinterval. 
This information seems disappointing for one who 
wishes to solve \eqref{eq-inveq_equiv} under 
simple differentiability condition. 
However, in our main theorems, 
we are going to show that if we directly assume 
$ \varphi $ to be a derivative, then the 
solutions are exactly the same as in the setting 
discussed by Kiss in \cite{Kis24}. 
As hinted above, this was a desired achievement, 
because it means 
that the same techniques which were used 
in \cite{Kis22} during the 
investigations of equation \eqref{eq-inveq_equiv} 
containing continuously differentiable functions 
can be applied in the situation when only 
differentiability is assumed. 

The mentioned main result of this paper is the 
characterization of the solutions of \eqref{eq-maineq} 
when $ \varphi $ is assumed to be measurable. 
Measurability is a common weak regularity assumption 
in the theory of functional equations. 
For many notable functional equations 
(e.g. Cauchy, Jensen, Pexider type equations) the 
measurability of the solution automatically 
implies very strong regularity, such as 
higher order differentiability, as it is 
thoroughly discussed in \cite{Kuc09}. 
We will see that for the alternative 
equation \eqref{eq-maineq} the situation is different, 
as the solutions can be rather irregular. 
Yet the set of problematic irregularities of $ \varphi $ 
is just a zero measure set, which is still tolerable 
for the applications. 
Numerous authors have addressed the question of 
measurability of the solutions of particular 
functional equations in different settings, 
such as Járai \cite{Jar86, Jar10}, 
Losonczi \cite{Los93}, Kochanek and Lewicki \cite{KL11}, 
Ger \cite{Ger18}. 
The interested reader should also get acquainted 
with the monograph \cite{Jar05} of Járai 
concerning the regularity of 
solutions of functional equations.

\section[Darboux solutions]{A solution consisting of irregular Darboux functions}

In this section 
we answer the open problem 
proposed by Kiss \cite{59ISFE} during the $59$th ISFE. 
Namely, we prove the existence of a 
nontrivial solution of \eqref{eq-maineq} consisting of 
such Darboux functions that are not constant on 
any nonempty open interval. 

Throughout the whole paper we consider 
the real line $ \RR $ equipped with the 
standard norm topology 
induced by the absolute value. 
The cardinal number 
continuum will be denoted by $ \cnc $. 
We will use the notation 
$ \RRx := \RR \cup \lbrace -\infty, +\infty \rbrace $ 
for the set of extended real numbers. 
If $ \emptyset \neq I \subseteq \RR $ is an interval, 
we say that a function 
$ f : I \map \RR $ has the 
Darboux property (or, simply, is Darboux) if 
the following assertion holds: 
for all $ x,y \in I $ such that 
$ x < y $ and $ f(x) \neq f(y) $, 
and for every 
$ \lambda \in 
\left] \, \min \lbrace f(x) , f(y) \rbrace \, , \, 
\max \lbrace f(x) , f(y) \rbrace \, \right[ \,$, 
there exists 
$ z \in \left] x , y \right[ $ for which 
$ f(z) = \lambda $ holds. 
We note that this is equivalent to 
the property that $ f(J) $ is an 
interval whenever $ J \subseteq I $ is an interval 
(see \cite[Theorem 9.2]{vRS82}). 

For any function 
$ f : I \map \RR $ the support of $ f $ 
will be denoted by $ \supp f $: 
\[
\supp f := 
\lbrace \, x \in I \ : \ f(x) \neq 0 \, \rbrace. 
\]

%
%

Sometimes we call a triplet a functions 
$ (\varphi \,, \psi_1 \,, \psi_2) $ a 
trivial solution of \eqref{eq-maineq} if 
$ \varphi(u) = 0 $ for all $ u \in J $ or 
$ \psi_1(x) = \psi_2(y) $ for all 
$ (x,y) \in I_1 \times I_2 \, $. 

As mentioned earlier, Kiss proposed the question whether 
the solutions of \eqref{eq-maineq} can be characterized in the 
same way as in \cite[Theorem 6]{Kis24}, when we assume that 
$ \varphi $ is Darboux (or even all three functions 
$ \varphi \,, \psi_1 \,, \psi_2 $ are Darboux). 
We proceed with an example that gives a negative answer 
for this question. 

\begin{thm}\label{thm-darboux_example}
The functional equation \eqref{eq-maineq} has a solution 
$ \left( \varphi \,, \psi_1 \,, \psi_2 \right) $ such that 
$ \varphi : \RR \map \RR $, 
$ \psi_1 : \RR \map \RR $ and 
$ \psi_2 : \RR \map \RR $ are Darboux functions, 
yet none of these functions are continuous 
at any $ x \in \RR $. 
\end{thm}

\begin{proof}
Let $ \dlth $ be a Hamel basis for $ \RR $, 
that is, $ \dlth $ is a basis of $ \RR $ 
if we consider it as a vector space over $ \QQ $. 
It is well-known 
(see \cite[Theorem 4.2.1.]{Kuc09}) that such basis 
$ \dlth $ exists, moreover we may assume that 
$ 1 \in \dlth $. 
As the cardinality of any Hamel basis 
is continuum, we may define a bijection 
$ \beta : \left] 0, 2 \right[ \map \dlth $ 
such that $ \beta(1) = 1 $. 
We split $ \dlth $ into two parts 
with the help of this bijection:  
\[
\dlth_1 := \beta \left( \, 
\left] 0,1 \right] \, \right) 
\hspace{10mm} \mbox{ and } \hspace{10mm} 
\dlth_2 := \beta \left( \, 
\left] 1,2 \right[ \, \right). 
\]
Clearly, $ 1 \in \dlth_1 $ and 
$ \card \dlth_1 = \card \dlth_2 = \cnc $. 
Let $ G_1 $ and $ G_2 $ be the linear hulls of 
$ \dlth_1 $ and $ \dlth_2 $ over $ \QQ $. 
That is, 
\[
G_i := 
\lbrace \, 
q_1 h_1 + \dots + q_n h_n \ : \ 
n \in \NN \mbox{ and } 
q_1 \,, \dots \,, q_n \in \QQ \,, 
h_1 \,, \dots h_n \,, \in \dlth_i 
\, \rbrace
\]
for $ i = 1,2 $. 
Due to the definition, $ G_i $ is in fact a 
vector space over $ \QQ $ with the basis $ \dlth_i $. 
In particular, $ \left( G_i \,, + \right) $ is an 
additive subgroup of $ \left( \RR , + \right) $ 
for $ i=1,2 $. 

\begin{claim}\label{clm-ekv1}
Let us define the relation 
$ \sim_1 $ on $ \RR $ 
in the following way: 
for any $ x,y \in \RR $ 
\[
x \sim_1 y \ \Longleftrightarrow \ x-y \in G_1 \, . 
\]
Then $ \sim_1 $ is an equivalence relation. 
Moreover, if $ b,h \in \dlth_2 $ 
are arbitrary then 
$ b \sim_1 h $ implies $ b = h $. 
\end{claim} 

\begin{prclm}
One can easily check that $ \sim_1 $ is indeed 
reflexive, symmetric and transitive, 
because $ \left( G_1 \,, + \right) $ is a group. 
Moreover, if $ b \sim_1 h $ then 
their difference is in $ G_1 \,$, 
hence there exist $ n \in \NN $ and 
$ q_1 \,, \dots \,, q_n \in \QQ $, 
$ h_1 \,, \dots \,, h_n \in \dlth_1 $ such that 
\[
b-h = q_1 h_1 + \dots + q_n h_n 
\] 
which means 
$ 0 = h - b + q_1 h_1 + \dots + q_n h_n \,$. 
Now $ h, b \in \dlth_2 $ so if 
$ b \neq h $ was true then this would mean that the set 
$ \lbrace h, b, h_1 \,, \dots \,, h_n \rbrace 
\subset \dlth $ 
is linearly dependent which is impossible, since 
$ \dlth $ is a basis. 
\hfill $ \boxtimes $ 
\end{prclm}

Since an equivalence relation always induces a partition on the 
original set, we will consider the equivalence classes 
of $ \sim_1 \,$. 
For any number $ x \in \RR $ its equivalence 
class will be denoted by $ [x]_1 \,$. 
We may note that $ [x]_1 \,$ is 
nothing else but the coset 
$ x + G_1 $ of the subgroup 
$ \left( G_1 \,, + \right) $ 
in the Abelian group $ \left( \RR \,, + \right) $. 
Let $ \dlta $ denote the set of all equivalence classes: 
\[
\dlta := \lbrace \, [x]_1 \, : \, x \in \RR \, \rbrace. 
\]
As assured in Claim \ref{clm-ekv1}, 
all the elements of $ \dlth_2 $ are 
in different equivalence classes, so 
$ \cnc = \card \dlth_2 \leq \card \dlta 
\leq \card \RR = \cnc $, hence 
$ \card \dlta = \cnc $. So there exists a bijection 
$ \gamma : \dlta \map \RR $. 

\begin{claim}\label{clm-Drbx1}
The function $ f : \RR \map \RR $  defined as 
\[
f(x) := \gamma \left( \, [x]_1 \, \right) 
\hspace{20mm} ( x \in \RR ) 
\] 
maps every nonempty open interval $ I \subseteq \RR $ 
onto $ \RR $. In particular, $ f $ has the 
Darboux property, but it $ f$ is not continuous 
at any point. 
\end{claim} 

\begin{prclm}
First of all, every class in $ \dlta $ is dense in $ \RR $. 
This follows from the fact that 
$ 1 \in \dlth_1 $ implies $ \QQ \subset G_1 \,$, 
thus, for all $ x \in \RR $, the inclusion 
$ [x]_1 = x + G_1 \supset x + \QQ $ holds while 
$ \QQ $ is dense. 

Now let $ \emptyset \neq I \subseteq \RR $ be 
any open interval and let 
$ \lambda \in \RR $ be arbitrary. 
As $ \gamma : \dlta \map \RR $ is a bijection, 
there exists $ w \in \RR $ such that 
$ \gamma \left( [w]_1 \right) = \lambda $. 
As we have established earlier in the proof, 
$ [w]_1 $ is dense in $ \RR $, 
so there exists $ z \in [w]_1 $ 
such that $ z \in I $. 
But then 
\[
f(z) = \gamma \left( [z]_1 \right) = 
\gamma \left( [w]_1 \right) = \lambda. 
\]
Since $ \lambda $ was arbitrary, 
we have obtained $ f(I) = \RR $. 
The fact that $ f $ is nowhere continuous but the 
has the Darboux property is an immediate consequence. 
\hfill $ \boxtimes $ 
\end{prclm}

We may note that the function 
$x \mapsto f(-x)$ is again a nowhere continuous 
Darboux function. 

\begin{claim}\label{clm-Drbx2}
There exists a nowhere continuous Darboux function 
$ \varphi : \RR \map\RR $ such that 
$ \supp \varphi = G_1 \,$. 
\end{claim} 

\begin{prclm} 
Let us consider a relation 
$ \sim_2 $ on $ G_1 $ defined by 
\[
x \sim_2 y \ \Longleftrightarrow \ x-y \in \QQ 
\hspace{15mm} \left( x,y \in G_1 \right). 
\]
It is trivial that 
$ \sim_2 $ is an equivalence relation, since 
$ \left( \QQ , + \right) $ is a subgroup of 
$ \left( G_1 \,, + \right) $. Consequently, the 
equivalence classes yield a partition for $ G_1 \,$. 
We may denote 
the equivalence class of an element 
$ x \in G_1 $ by $ [x]_2 \,$, 
and this class is simply the coset $ x + \QQ $ 
in the Abelian group $ \left( G_1 \,, + \right) $. 
Therefore every equivalence class is dense in $ \RR $. 

We may also claim that no equivalence class can contain 
two different elements from $ \dlth_1 $. Indeed, 
$ h,b \in \dlth_1 \,$, $ h \neq b $ and $ h \sim_2 b $ 
would mean that there exists 
$ q \in \QQ $ such that 
$ h-b = q $, so $ 0 = b - h + q $. 
In other words, 
$ \lbrace 1, b, h \rbrace \subset 
\dlth_1 \subset \dlth $ is a 
linearly dependent system, 
which contradicts that $ \dlth $ is a Hamel basis. 

$ \dltb $ will denote the set of 
equivalence classes, i.e. 
$ \dltb := \lbrace \, [x]_2 \, : \, x \in G_1 \, \rbrace $. 
Since all elements from $ \dlth_1 $ are 
in different classes, we have 
$ \cnc = \card \dlth_1 \leq \card \dltb \leq \cnc $ 
which ensures $ \card \dltb = \cnc $. 
Hence there exists a bijection 
$ \delta : \dltb \map \left] 0,1 \right[ \, $. 

Let us define an auxiliary function 
$ \psi : G_1 \map \left] 0,1 \right[ $ 
with the formula 
\[
\psi(x) := \delta \left( \, [x]_2 \, \right) 
\hspace{15mm} \left( x \in G_1 \right). 
\]
Given this $ \psi $ we may construct 
$ \varphi : \RR \map \RR$ as 
\[
\varphi(x) := 
\begin{cases}
\psi(x) \qquad & \mbox{ if } x \in G_1 \\ 
0 \qquad & \mbox{ if } x \in \RR \setminus G_1 \,. 
\end{cases}
\]
Clearly, $ \supp \varphi = G_1 \,$, 
as the range of $ \delta $ is $ \left] 0,1 \right[ \, $, 
so $ \varphi(x) = \psi(x) \neq 0 $ 
for all $ x \in G_1 \,$. 
Now let $ I \subseteq \RR $ be any nonempty open interval 
and $ \lambda \in \left] 0,1 \right[ \,$ be also 
arbitrary. 
Now there exists $ w \in G_1 \,$ such that 
$ \delta \left( [w]_2 \right) = \lambda $. 
But $ [w]_2 $ is dense in $ \RR $, so there exists 
$ z \in I \cap G_1 $ such that 
$ z \in [w]_2 $ which means 
\[
\varphi(z) = \psi(z) = 
\delta \left( [z]_2 \right) = 
\delta \left( [w]_2 \right) = \lambda. 
\]
This shows that 
$ \left] 0,1 \right[ \subseteq \varphi(I) $. 
Together with the fact that 
the range of $ \varphi $ is $ \left[ 0,1 \right[ \,$, 
this implies the Darboux property of $\varphi$. 
Moreover, $ \varphi $ cannot be continuous at any point 
$ x_0 \in \RR $, 
because in that case it would map an open neighborhood 
of $ x_0 $ into an interval with length less than $ 1 $. 
\hfill $ \boxtimes $   
\end{prclm}

Now let us define 
$ \psi_1 \,, \psi_2 : \RR \map \RR $ 
as  
\[
\psi_1(x) := f(x) 
\qquad \mbox{ and } \qquad
\psi_2(x) := f(-x) \hspace{10 mm} (x \in \RR) 
\]
where $ f $ is the function 
appearing in Claim \ref{clm-ekv1}. 
We proceed to show that 
$ \left( \varphi \,, \psi_1 \,, \psi_2 \right) $ 
is a solution of \eqref{eq-maineq}. 
We distinguish two cases. 
Firstly, suppose that $ x,y \in \RR $ 
are such that 
$ \frac{x+y}{2} \in \RR \setminus G_1 \,$. 
Then the equation \eqref{eq-maineq} 
is obviously fulfilled, since 
$ \varphi \left( \frac{x+y}{2} \right) = 0 $. 
The second possibility is when 
$ \frac{x+y}{2} \in G_1 \,$. That is, 
$ x+y = 2g $ for some $ g \in G_1 $ 
which is equivalent to 
$ x - (-y) = 2g \in G_1 \,$. 
This means $ x \sim_1 -y $ whence 
$ [x]_1 = [-y]_1 $. Consequently, 
\[
\psi_1(x) = f(x) = \gamma \left( [x]_1 \right) = 
\gamma \left( [-y]_1 \right) = 
f( -y ) = \psi_2(y) \,, 
\]
which implies 
$ \psi_1(x) - \psi_2(y) = 0 $, so \eqref{eq-maineq} 
is again fulfilled. 
Summarizing the two cases, we may conclude that, 
for every $ x \in \RR $ 
and for every $ y \in \RR $, 
the functional equation \eqref{eq-maineq} holds. 
Finally, let us recall that none of the three functions 
are continuous at any point, 
as we have established in 
Claims \ref{clm-Drbx1} and \ref{clm-Drbx2}.
\end{proof}


\section[Measurable solutions]{Solutions under measurability condition}

As we will see in the sequel, the irregular behavior 
of the previous solution was due to the fact that 
$ \supp \varphi $ was either a set 
with Lebesgue measure $ 0 $, 
or it is possible that $ \varphi $ was not even measurable 
(depending on the measurability properties of 
the Hamel basis $ \dlth $). 

In this section we characterize the solutions of 
\eqref{eq-maineq} under the assumption that 
$ \varphi $ is Lebesgue measurable. 
When we claim that a triplet 
$ \left( \varphi \,, \psi_1 \,, \psi_2 \right) $ 
is a solution, we mean it in such a way that 
there are appropriate nonempty open intervals 
$ I_1 \,, I_2 $ and $ J = \frac{1}{2} (I_1 + I_2) $ 
so that 
$ \varphi : J \map \RR $, 
$ \psi_1 : I_1 \map \RR $ and 
$ \psi_2 : I_2 \map \RR $. In most cases 
we do not emphasize these domains explicitly. 
From now on, by measurability we 
mean Lebesgue measurability, 
and the Lebesgue measure on $ \RR $ 
will be denoted by $ \mu $. 

During our investigations 
we will often consider the reflection of a subset 
$ S \subseteq \RR $ with respect to 
another subset $ \emptyset \neq M \subseteq \RR $. 
For this purpose let us introduce a notation: 
\[
\rfl{S}{M} := 
\lbrace
2m - x \, \vert \, m \in M \mbox{ and } x \in S 
\rbrace. 
\] 
Using the pointwise multiplication and the 
Minkowski sum of sets we have $ \rfl{S}{M} = 2 M - S $. 
In the particular case when 
$ M = \lbrace m \rbrace $ is a singleton 
we will write 
$ \rfl{S}{m} $ instead of $ \rfl{S}{ \lbrace m \rbrace } $. 
Similarly, for 
$ s,m \in \RR $ we will use the notation 
$ \rfl{s}{m} = 2m-s $. 

We mention that a slightly more general version of 
this reflection operator is introduced in 
\cite{Kis24} as well, 
where plenty of its properties are discussed. 

At this point we formulate an elementary statement 
concerning sets of $ \RR $ with positive Lebesgue measure, 
which will be used frequently later on. 
\begin{prop}\label{prop-measure}
Let $ A \subseteq \RR $ be a Lebesgue-measurable set 
with $ \mu (A) > 0 $. 
Then there exists $ s \in A $ such that 
for all $\delta > 0$ it holds that 
\[
\mu \left( \, A \cap \left[ s, s+\delta \right] \, \right) > 0. 
\]
\end{prop}

\begin{proof}
In fact much more is true, since due to the 
Lebesgue Density Theorem 
(see \cite[Theorem 21.29]{vRS82}), the Lebesgue density is 
$ 1 $ almost everywhere on $ A $. That is, 
\[
\lim_{\delta \to 0+} 
\frac{\mu \left( \, 
A \cap \left[ s - \delta, s + \delta \right] \, \right)}
{2 \delta} 
= 1 
\]
holds for almost every $ s \in A $. 
For any such number $ s $ the inequality 
$ \mu \left( \, A \cap 
\left[ s, s+\delta \right] \, \right) > 0 $ must 
be fulfilled for all $ \delta > 0 $, 
otherwise the Lebesgue density in that point could 
be at most $ \frac{1}{2} \,$. 
\end{proof}

We will use this statement in the particular case when 
$ A = \supp \varphi $. 

\begin{coro}\label{coro-positive_measure}
Let 
$ \left( \varphi \,, \psi_1 \,, \psi_2 \right) $ 
be a solution of \eqref{eq-maineq} such that 
the function $ \varphi $ is measurable. 
Then 
$ \mu \left( \supp \varphi \right) > 0 $ 
implies that there exists 
$ s \in \supp \varphi $ such that 
\[
\mu \left( \, \supp \varphi \cap 
\left[ s, s + \delta \right] \, \right) > 0 
\hspace{10 mm} 
\mbox{ for all } \delta > 0. 
\] 
\end{coro}

The next lemma will play a crucial role later on. 
It turns out that the measurability provides us a tool 
which helps us to replace the condition for 
the set of zeros of $ \varphi $ utilized in \cite{Kis24}. 
This tool is a classical 
theorem of Steinhaus \cite{Ste20}. 

\begin{lemm}\label{lemm-locally_constant}
Let 
$ \left( \varphi \,, \psi_1 \,, \psi_2 \right) $ 
be a solution of \eqref{eq-maineq}. 
Suppose that $ \varphi $ is measurable, 
and $ s \in J $ is such a point that 
$ \mu \left( \, \supp \varphi \cap 
\left[ s, s + \delta \right] \, \right) > 0 $ 
for every $ \delta > 0 $. 
Then, if $ \emptyset \neq K \subseteq I_1 $ 
is an interval fulfilling 
$ \rfl{K}{s} \subseteq I_2 \, $, then 
there exists $ \lambda \in \RR $ for which 
\[
\psi_1(x) = \psi_2(y) = \lambda 
\qquad \mbox{ holds for all } 
x \in K \mbox{ and } y \in \rfl{K}{s}. 
\]
\end{lemm}

\begin{proof}
Let 
$ x,y \in K $ be arbitrary numbers, 
fulfilling $ x < y $. 
Now 
\[ 
\rfl{ \, [x,y] \, }{s} = 
\left[ 2s-y, 2s-x \right] 
\subset \rfl{K}{s} \subseteq I_2 \,, 
\]
using the definition of the reflection operator, 
and the assumption for $ K $. 
Considering that $ I_2 $ is open, 
there exists $ d > 0 $ such that 
$ 2s - x + 2d \in I_2 \,$, so 
\[
\rfl{ \, [x,y] \, }{ \, [s, s+d] \, } = 
\left[ 2s-y , 2s + 2d -x \right] \subseteq I_2 
\]
holds. 
On the other hand, we have supposed 
$ \mu \left( \, \supp \varphi \cap 
\left[ s, s + d \right] \, \right) > 0 $. 
Therefore, applying the 
theorem of Steinhaus, the difference set 
\[
\Delta := \lbrace \, t_1 - t_2 \ : \ 
t_1 \,, t_2 \in \supp \varphi \cap [s,s+d] \, \rbrace
\] 
contains an open neighborhood of $ 0 $, 
i.e. there exists 
$ \varepsilon > 0 $ such that 
$ \left] - \varepsilon , \varepsilon 
\right[ \subseteq \Delta $. 
Now let $ m \in \NN $ be such that 
\[
h := \frac{y-x}{2m} \, \in \, 
\left] \, 0, \varepsilon \, \right[ 
\, \mbox{ is fulfilled}.
\]
Then there exist 
$ s_1 \,, s_2 \in \supp \varphi \cap [s, s+d] $ 
such that 
$ s_2 = s_1 + h $. 
Let us introduce new notations 
for some particular points: 
\[
x_j := x + j \cdot 2h \,, \qquad 
\widehat{x}_j := \rfl{x_j}{s_1} \ \mbox{ and } \ 
\widehat{\widehat{x}}_j := \rfl{\widehat{x}_j}{s_2} 
\qquad (j = 0,1, \dots , m).
\]
Obviously, 
$ x_0 = x $, $ x_m = y $ and $ x_j \in [x,y] $. 
Therefore 
$ \widehat{x}_j \in 
\rfl{\, [x,y] \,}{\, [s,s+d] \,} \subseteq I_2 $ 
for $ j=0,1, \dots ,m $ while 
\[
\widehat{\widehat{x}}_j = 2s_2 - \widehat{x}_j = 
2s_2 - 2s_1 + x_j = x_j + 2h = x_{j+1} \in I_1 
\qquad \mbox{for } j = 0,1, \dots , m-1. 
\] 
At this point let us utilize the fact that 
$ \left( \varphi \,, \psi_1 \,, \psi_2 \right) $ 
is a solution of \eqref{eq-maineq}. 
In particular, 
\begin{align*}
\varphi \left( 
\frac{x_j + \widehat{x}_j}{2} \right) 
\bigl( 
\psi_1 (x_j) - \psi_2(\widehat{x}_j) 
\bigr) &= 0 
\qquad (j = 0,1, \dots m) \ \mbox{ and } \\ 
\varphi \left( 
\frac{\widehat{\widehat{x}}_j + \widehat{x}_j}{2} \right) 
\bigl( 
\psi_1 (\widehat{\widehat{x}}_j) - \psi_2(\widehat{x}_j) 
\bigr) &= 0 
\qquad (j = 0,1, \dots m-1), 
\end{align*} 
which is equivalent to 
\begin{align*} 
\varphi \left( s_1 \right) 
\bigl( 
\psi_1 (x_j) - \psi_2(\widehat{x}_j) 
\bigr) &= 0 
\qquad (j = 0,1, \dots m) \ \mbox{ and } \\ 
\varphi \left( s_2 \right) 
\bigl( 
\psi_1 (x_{j+1}) - \psi_2(\widehat{x}_j) 
\bigr) &= 0 
\qquad (j = 0,1, \dots m-1).  
\end{align*} 
As $ s_1 \,, s_2 \in \supp \varphi $, 
these equations imply 
\[ 
\psi_1(x_j) = \psi_2( \widehat{x}_j ) = \psi_1(x_{j+1}) 
\qquad \mbox{ for every } 
j = 0,1, \dots ,m-1.
\]
Therefore 
$ \psi_1(x) = \psi_1 (x_0) = \psi_1 (x_1) = \dots = 
\psi_1 (x_m) = \psi_1(y) $ 
holds as well. 
Since $ x,y \in K $ were arbitrary, it follows 
that there exists 
$ \lambda \in \RR $ such that 
$ \psi_1(x) = \lambda$  for all $ x \in K $. 

Finally, we have to verify that 
$ \psi_2 $ is also constant on $ \rfl{K}{s} $ 
with value $ \lambda $. 
For this purpose let observe that 
$ \left( \varphi \,, \psi_1 \,, \psi_2 \right) $ 
is a solution of \eqref{eq-maineq} if, 
and only if, 
$ \left( \varphi \,, \psi_2 \,, \psi_1 \right) $ 
is a solution. 
Thus let us apply the same proof we have just carried out, 
but now for the solution 
$ \left( \varphi \,, \psi_2 \,, \psi_1 \right) $, 
and with the starting interval 
$ \rfl{K}{s} \subseteq I_2 \,$. 
Since 
$ \rfl{\rfl{K}{s}}{s} = K \subseteq I_1 \,$, 
the assumptions of the lemma are fulfilled 
for the solution 
$ \left( \varphi \,, \psi_2 \,, \psi_1 \right) $. 
Hence, according to the first part of our proof, 
there exists $ \nu \in \RR $ such that 
$ \psi_2(y) = \nu $ for all 
$ y \in \rfl{K}{s} $. 
But, for the 
points 
$ x_0 = x \in K $ and 
$ \widehat{x}_0 \in \rfl{K}{s}$ 
introduced before, we had 
\[
\lambda = \psi_1(x_0) = \psi_2(\widehat{x}_0) = \nu \,, 
\]
which completes the proof. 
\end{proof}

The next proposition claims that nontrivial solutions 
are locally constant on some ends of their domains. 

\begin{prop}\label{prop_end_constant}
Let 
$ \left( \varphi \,, \psi_1 \,, \psi_2 \right) $ 
be a solution of \eqref{eq-maineq} such that 
$ \varphi $ is measurable. 
Suppose that there exist 
$ \gamma_1 \in I_1 \,$, 
$ \gamma_2 \in I_2 $ such that for 
$ s = \frac{\gamma_1 + \gamma_2}{2} \in J $ 
it holds that 
\[
\mu \left( \supp \varphi \cap 
\left[ s, s + \delta \right] \right) > 0 
\hspace{10 mm} 
\mbox{ for all } \delta > 0. 
\] 
Then exactly one of the following three cases holds. 
\begin{itemize}
\item[Case 1.] 
There exist $ \lambda \in \RR $ and 
nonempty open intervals 
$ P_i \subsetneqq I_i $ such that 
$ \gamma_i \in P_i $ and $ \inf P_i = \inf I_i $ 
($ i=1,2 $) 
moreover 
\[
\psi_1(x) = \psi_2(y) = \lambda 
\qquad \mbox{ holds for all } 
x \in P_1 \mbox{ and } y \in P_2 \,. 
\]
\item[Case 2.]
There exist $ \lambda \in \RR $ and 
nonempty open intervals 
$ Q_i \subsetneqq I_i $ 
such that 
$ \gamma_i \in Q_i $ and $ \sup Q_i = \sup I_i $ 
($ i=1,2 $) moreover 
\[
\psi_1(x) = \psi_2(y) = \lambda 
\qquad \mbox{ holds for all } 
x \in Q_1 \mbox{ and } y \in Q_2 \,. 
\]
\item[Case 3.]
$ \psi_1 $ is constant on $ I_1 $ or 
$ \psi_2 $ is constant on $ I_2 $. 
\end{itemize}
\end{prop} 

\begin{proof}
Let us define the sets 
\[
C_1 := \lbrace \, 
x \in I_1 \, : \, \rfl{x}{s} \in I_2 
\, \rbrace 
\hspace{5mm} \mbox{ and } \hspace{5mm} 
C_2 := \lbrace \, 
y \in I_2 \, : \, \rfl{y}{s} \in I_1 
\, \rbrace. 
\]
Since $ I_1 \,$, $ I_2 $ are open intervals, 
it is easy to see that 
$ C_1 \,$, $ C_2 $ are also open intervals and, 
clearly, 
$ \gamma_1 \in C_1 \,$, $ \gamma_2 \in C_2 \,$. 
We denote the endpoints of the intervals as 
\[
\alpha_i = \inf C_i \in \RRx 
\hspace{5mm} \mbox{ and } \hspace{5mm} \ 
\beta_i = \sup C_i \in \RRx 
\qquad (i=1,2). 
\]
From the definition of $ C_1 $ and $ C_2 $ it is 
straightforward that $ \rfl{C_1}{s} = C_2 $ and 
$ \rfl{C_2}{s} = C_1 \,$. 
In particular, 
$ \rfl{\alpha_1}{s} = \beta_2 $ and 
$ \rfl{\beta_1}{s} = \alpha_2 \,$, using the 
convention that 
$ \rfl{+ \infty}{s} = - \infty $ and 
$ \rfl{- \infty}{s} = + \infty $. 
 
We point out that that 
$ \inf I_1 < \alpha_1 $ and $ \sup I_2 > \beta_2 $ 
cannot occur simultaneously. 
Assume, on the contrary, that this is not the case. 
Firstly, this would imply 
$ \alpha , \beta \in \RR $. 
Now, since $ I_1 \,$, $ I_2 $ are open and 
$ \rfl{\alpha_1}{s} = \beta_2 \,$, 
there exists $ \varepsilon > 0 $ such that 
\[
x_1 := \alpha_1 - \varepsilon \in I_1 
\ \mbox{ and } \ 
\rfl{x_1}{s} = 2s - \alpha_1 + \varepsilon = 
\beta_2 + \varepsilon \in I_2 \,, 
\]
so 
$ x_1 \in C_1 $ while 
$ x_1 < \alpha_1 = \inf C_1 $ 
which is a contradiction. 
For similar reasons, 
$ \inf I_2 < \alpha_2 $ and $ \sup I_1 > \beta_1 $ 
cannot occur simultaneously, either. 

Finally, recalling 
$ \rfl{C_1}{s} = C_2 $ and using 
Lemma \ref{lemm-locally_constant}, 
there exists a constant $ \lambda \in \RR $ 
such that 
\[
\psi_1(x) = \psi_2(y) = \lambda
\qquad \mbox{ holds for all } 
x \in C_1 \mbox{ and } y \in C_2 \,. 
\]
Now there are three possibilities. 
\begin{enumerate}
\item[(i)] 
$ C_i = I_i $ for at least one 
$ i \in \lbrace 1,2 \rbrace $, 
which yields Case 3. of the statement. 
\item[(ii)] 
$ \alpha_1 = \inf I_1 $ but 
$ \beta_1 < \sup I_1 \,$. 
We have seen that the latter inequality implies 
$ \alpha_2 = \inf I_2 \,$. 
Now $ \beta_2 = \sup I_2 $ would 
result in the previous case with 
$ C_2 = I_2 \,$, so we consider 
$ \beta_2 < \sup I_2 \,$. 
Using Lemma \ref{lemm-locally_constant} 
we get Case 1. of our statement, 
since $ \psi_1 $ and $ \psi_2 $ 
are constant functions with the same value 
$ \lambda $ on 
\[
P_1 := C_1 = \left] \, \alpha_1 \,, \beta_1 \, \right[ 
\subsetneqq I_1 
\hspace{3mm} \mbox{ and } \hspace{3mm} 
P_2 := C_2 = \left] \, \alpha_2 \,, \beta_2 \, \right[ 
\subsetneqq I_2 \,, 
\mbox{ respectively.}
\] 
\item[(iii)] 
$ \inf I_1 < \alpha_1 \,$, which implies 
$ \sup I_2 = \beta_2 \,$. 
Again, to avoid $ C_2 = I_2 \,$, 
we restrict ourselves to 
$ \inf I_2 < \alpha_2 $ and thus 
$ \sup I_1 = \beta_1 \,$. 
This way we get Case 2. of the proposition, 
because Lemma \ref{lemm-locally_constant} 
ensures that 
$ \psi_1 $ and $ \psi_2 $ 
are constant functions with the same value 
$ \lambda $ on the intervals 
\[
Q_1 := C_1 = \left] \, \alpha_1 \,, \beta_1 \, \right[ 
\subsetneqq I_1 
\hspace{3mm} \mbox{ and } \hspace{3mm} 
Q_2 := C_2 = \left] \, \alpha_2 \,, \beta_2 \, \right[ 
\subsetneqq I_2 \,, 
\mbox{ respectively.} 
\]
\end{enumerate}
\end{proof}

Now we have all the auxiliary tools to describe the 
structure of the possible solutions of \eqref{eq-maineq} 
under the assumption that $\varphi$ is measurable. 
In the next theorems some open intervals of the form 
$ \left] \, \alpha , \beta \, \right[ \subseteq \RR $ 
will appear where 
$ \alpha, \beta \in \RRx $ and $ \alpha \leq \beta $. 
We will use the convention that if 
$ \alpha = \beta $ then automatically 
$ \left] \, \alpha , \beta \, \right[ = \emptyset \,$ 
even if $ \alpha, \beta $ are extended real numbers. 

\begin{thm}\label{thm-main_measurable}
Let 
$ \left( \varphi \,, \psi_1 \,, \psi_2 \right) $ 
be a solution of \eqref{eq-maineq} 
and suppose that 
$ \varphi : J \map \RR $  is measurable. 
Then exactly one of the following statements is true. 
\begin{enumerate} 
\item 
$ \mu \left( \supp \varphi \right) = 0 $ 
or there exists 
$ \lambda \in \RR $ such that 
$ \psi_1(x) = \psi_2(y) = \lambda $ 
for all 
$ x \in I_1 $ and $ y \in I_2 \,$. 
\item 
There exist constants 
$ \lambda, \nu \in \RR $ 
and, for $ i = 1,2 $, there exist 
$ a_i \,, b_i \in \RRx $ 
such that 
$ \inf I_i \leq a_i \leq b_i \leq \sup I_i $ 
and for the open intervals 
\[
U_i := \left] \, \inf I_i \,, a_i \, \right[ 
\qquad \mbox{ and } \qquad 
V_i := \left] \, b_i \,, \sup I_i \, \right[ 
\]
the following assertions are fulfilled: 
\begin{itemize}
\item 
$ U_i \neq I_i $ and $ V_i \neq I_i $ for $ i = 1,2 $ ; 
\item 
$ U_1 \,, U_2 $ are both nonempty or 
$ V_1 \,, V_2 $ are both nonempty; 
\item 
$ \psi_1 (x) =  \psi_2 (y) = \lambda $ for all 
$ (x,y) \in U_1 \times U_2 $ and 
$ \psi_1 (x) =  \psi_2 (y) = \nu $ for all 
$ (x,y) \in V_1 \times V_2 $ ; 
\item  
\[
\mu \left( \left( 
\frac{1}{2} \left( K_1 + I_2 \right) \cup 
\frac{1}{2} \left( I_1 + K_2 \right) \right)
\cap \supp  \varphi \right) = 0 
\] 
where 
$ K_i := I_i \setminus (U_i \cup V_i) $ 
is a nonempty subinterval 
which is closed in $ I_i $ ($ i=1,2 $). 
\end{itemize}

\item There exist a constant 
$ \lambda \in \RR $ and an index 
$ j \in \lbrace 1,2 \rbrace $ such that 
$ \psi_j(z) = \lambda $ for all $ z \in I_j \,$, 
while there exists a finite or countable collection of 
disjoint, nonempty open intervals 
$ U_n  \subset I_i $ 
($ i \neq j $ and $ n = 1, \dots, N $ 
where $ N \in \NN $ or $ N = \infty $) 
with the following properties: 
\begin{align*}
& \bigcup_{n=1}^N U_n \neq I_i 
\ \mbox{ and } \ \psi_i(z) = \lambda 
\ \mbox{ for all } \ z \in \bigcup_{n=1}^N U_n \,, 
\mbox{ moreover } \\  
& \varphi(t) = 0 \ \mbox{ for all } \ 
t \in \frac{1}{2} \left( 
\left( I_i \setminus \bigcup_{n=1}^N U_n \right) + I_j 
\right).  
\end{align*}
\end{enumerate} 
\end{thm}

\begin{proof}
It is easy to see that the first case can indeed occur, 
for instance when the solution is trivial. 
However, from now on, assume that \emph{(1)} 
does not hold. 
In particular, 
$ \mu ( \supp \varphi ) > 0 $, 
so we may use Corollary \ref{coro-positive_measure} 
to see that the assumptions of 
Proposition \ref{prop_end_constant} are fulfilled. 

We start our investigations in the situation when 
Case 1. or Case 2. holds in 
Proposition \ref{prop_end_constant}. 
That is, neither $ \psi_1 $ nor $ \psi_2 $ is constant. 
For any $ \lambda \in \RR $ let us define the sets 
\begin{align*}
A_i \left( \lambda \right) & := 
\lbrace \, 
t \in I_i \ : \ \psi_i(z) = \lambda 
\mbox{ for all } z \in \, 
\left] \,  \inf I_i \,, t \, \right[ \, 
\rbrace 
\hspace{5mm} \mbox{ and } \\ 
B_i \left( \lambda \right) & := 
\lbrace \, 
t \in I_i \ : \ \psi_i(z) = \lambda 
\mbox{ for all } z \in \, 
\left] \, t , \sup I_i \, \right[ \, 
\rbrace 
\hspace{10mm} (i=1,2). 
\end{align*} 
Obviously, such a set can be empty 
but when it is not empty 
then it is an interval.
Furthermore, if 
$ A_i( \lambda ) \neq \emptyset $ 
(or, respectively, $ B_i(\lambda) \neq \emptyset $) 
for some 
$ \lambda \in \RR $ then 
$ A_i( \kappa ) = \emptyset $ 
(respectively, $ B_i( \kappa ) = \emptyset $), 
for every $ \kappa \neq \lambda $. 
Since $ \psi_1 \,, \psi_2 $ are not constant functions, 
$ A_i (\lambda) \neq I_i $ and 
$ B_i (\lambda) \neq I_i $ for $ i = 1,2 $ and 
for all $ \lambda \in \RR $. 
Now we will determine the sets 
$ U_i \,, V_i $ appearing in 
case \emph{(2)} of our theorem. 
Let 
\[
U_i := 
\begin{cases}
\left] \, \inf I_i \,, \sup A_i(\lambda) \, \right[ 
\ & \mbox{ if there exists } \lambda \in \RR 
\mbox{ such that } A_i(\lambda) \neq \emptyset \\ 
\emptyset 
\ & \mbox{ if } A_i(\lambda) = \emptyset 
\ \mbox{ for all } \lambda \in \RR 
\end{cases} 
\hspace{10 mm} (i=1,2). 
\]
Similarly, let 
\[
V_i := 
\begin{cases}
\left] \, \inf B_i(\nu) \,, \sup I_i \, \right[ 
\ & \mbox{ if there exists } \nu \in \RR 
\ \mbox{ such that } B_i(\nu) \neq \emptyset \\ 
\emptyset 
\ & \mbox{ if } B_i(\nu) = \emptyset 
\ \mbox{ for all } \nu \in \RR 
\end{cases} 
\hspace{10 mm} (i=1,2). 
\]
According to the previous observation 
on the uniqueness of the constants 
$ \lambda, \nu $, these sets are well-defined. 
Due to Proposition \ref{prop_end_constant}, 
$ U_1 \neq \emptyset $ and $ U_2 \neq \emptyset $, or 
$ V_1 \neq \emptyset $ and $ V_2 \neq \emptyset $ 
surely holds (possibly at the same time). 
The definition of $ U_i \,, V_i $ guarantees that 
$ \psi_i $ is constant $ \lambda $ on $ U_i $ and 
$ \psi_i $ is constant $ \nu $ on $ V_i $ for 
$ i = 1,2 $. 

To conclude case \emph{(2)} 
we have to show that 
\[
\mu \left( \left( 
\frac{1}{2} \left( K_1 + I_2 \right) \cup 
\frac{1}{2} \left( I_1 + K_2 \right) \right)
\cap \supp  \varphi \right) = 0 
\]
Assume that the contrary is true, thus either 
\[ 
\mu \left( 
\frac{1}{2} \left( K_1 + I_2 \right) \cap 
\supp \varphi \right) > 0 
\hspace{5mm} \mbox{ or } \hspace{5mm} 
\mu \left( 
\frac{1}{2} \left( I_1 + K_2 \right) \cap 
\supp \varphi \right) > 0 
\]
holds. We only discuss the situation when 
$ \mu \left( 
\frac{1}{2} \left( K_1 + I_2 \right) \cap 
\supp \varphi \right) > 0 $. 
From the other one we could derive a contradiction analogously. 
In this case, using 
Proposition \ref{prop-measure}, there exists 
$ s_0 \in \frac{1}{2} \left( K_1 + I_2 \right) 
\cap \supp \varphi $ 
such that 
\[
\mu \left( 
\frac{1}{2} \left( K_1 + I_2 \right) 
\cap \supp \varphi \cap 
\left[ \, s_0 \,, s_0 + \delta \, \right] \right) > 0 
\hspace{10 mm} 
\mbox{ for all } \delta > 0. 
\]
That is, there exist 
$ c_1 \in K_1 $ and $ c_2 \in I_2 $ such that 
$ s_0 = \frac{1}{2} \left( c_1 + c_2 \right) $. 
Hence the assumptions of Proposition \ref{prop_end_constant} 
are fulfilled for the solution, but 
now with the points 
$ c_1 \in K_1 \,$, $ c_2 \in I_2 $ and $ s_0 \in J $. 
As we have assumed that neither of 
$ \psi_1 $ and $ \psi_2 $ is constant, 
necessarily Case 1. or Case 2. is relevant 
in Proposition \ref{prop_end_constant}. 
Either way, $ \psi_1 $ is constant on 
an interval that contains $ c_1 $ and 
extends until one of the endpoints of $ I_1 \,$. 
Consequently, 
$ c_1 \in U_1 $ or $ c_1 \in V_1 $ is fulfilled 
which contradicts the fact that 
$ c_1 \in K_1 = 
I_1 \setminus \left( U_1 \cup V_1 \right) $. 

Finally, from the definitions of $ U_i $ and $ V_i \,$, 
together with the condition that neither 
$ \psi_1 $ nor $ \psi_2 $ is constant, 
we get that $ K_i \neq \emptyset $. 
Moreover, as it is obtained after the 
removal of open intervals 
$ U_i $, $ V_i $ from the two ends of 
$ I_i $, it follows that 
$ K_i $ is closed in $ I_i $ ( $ i=1,2 $). 
Thus we have verified that 
$ U_1 \,, U_2 \,, V_1 \,, V_2 \,, K_1 \,, K_2 $ 
indeed possess the properties 
assured in case \emph{(2)} of our theorem. 

The only remaining situation is when there exist 
$ \lambda \in \RR $ and an index 
$ j \in \lbrace 1,2 \rbrace $ such that 
$ \psi_j(z) = \lambda $ for all 
$ z \in I_j \,$, while 
at the same time $ \psi_i $ 
is not constant on $ I_i $ 
($ i \neq j $). 
We proceed with analyzing the structure of 
$ \rfl{I_j}{\supp \varphi} \cap I_i \,$. 
First of all, 
$ I_j $ is open, so the reflection 
$ \rfl{I_j}{\supp \varphi} = 
2 \cdot \supp \varphi - I_j $ 
is also open in $ \RR $, thus 
$ \rfl{I_j}{\supp \varphi} \cap I_i $ 
is open as well. 
Hence, according to the 
Representation theorem for open sets of $ \RR $, 
it is the union of a countable collection of 
nonempty, disjoint open intervals 
(see \cite[Theorem 3.11]{Apo81}). 
That is, there exist open intervals 
$ U_n \subseteq \RR $ 
($ n = 1, \dots ,N $ where 
$ N \in \NN $ or $ N = \infty $) 
such that 
$ U_k \cap U_l = \emptyset $ for all indices 
$ k \neq l $, and 
\[
\rfl{I_j \, }{ \, \supp \varphi } \cap I_i = 
\bigcup_{n = 1}^N U_n \,. 
\]
Of course this implies $ U_n \subseteq I_i \,$. 
Now let 
$ m \in \lbrace 1, \dots, N \rbrace $ be 
an arbitrary index, and $ z \in U_m \,$. 
Since 
$ z \in \rfl{I_j}{\supp \varphi } \cap I_i \,$, 
there exist 
$ s \in \supp \varphi $ and $ w \in I_j $ such that 
$ z = 2s-w $. 
Putting $ z $ and $ w $ into \eqref{eq-maineq} 
and using that $ \varphi (s) \neq 0 $, 
we get 
$ \psi_i (z) = \psi_j(w) = \lambda $. 
So $ \psi_i $ is indeed constant on 
$ \bigcup U_n $ with value $ \lambda $. 

For the last step let us choose an arbitrary 
$ t \in \frac{1}{2} \bigl( 
\left( I_i \setminus \bigcup U_n \right) + I_j 
\bigr) $. Then there exist 
$ z \in \left( I_i \setminus \bigcup U_n \right) $ 
and 
$ w \in I_j $ such that 
$ t = \frac{1}{2}(z+w) $, which means 
$ z = 2t - w $. But if 
$ t \in \supp \varphi $ then we have 
$ z \in \rfl{I_j}{\supp \varphi} \cap I_i = \bigcup U_n \,$, 
opposed to 
$ z \in \left( I_i \setminus \bigcup U_n \right) $. 
Therefore $ \varphi(t) = 0 $ must hold for every 
$ t \in \frac{1}{2} \bigl( 
\left( I_i \setminus \bigcup U_n \right) + I_j 
\bigr) $. 
Thus we have concluded case \emph{(3)} of the theorem. 
\end{proof}

\section{Solutions when $ \varphi $ is a derivative}

We wish to emphasize that 
Theorem \ref{thm-main_measurable} is 
not an "if and only if" type statement. 
This happens because for the reverse implication 
we could only guarantee that 
$ \varphi $ is $ 0 $ almost everywhere on some intervals, 
which is insufficient to automatically produce a solution. 
Therefore we are still a step away from a characterization 
theorem analogous to \cite[Theorem 6.]{Kis24} by Kiss. 
However, we have already mentioned 
in the Introduction, that for the applications 
the meaningful condition for $ \varphi $ 
would be that it is the derivative 
of a differentiable function. 
We know that a derivative is measurable, 
so Theorem \ref{thm-main_measurable} 
gives the three possible forms of a solution. 
What is more, 
under the assumption that $ \varphi $ is a derivative, 
we can prove the converse of our theorem. 
The reason for that is the following statement. 

\begin{lemm}\label{lemm-derivative_ae00}
Let $ \emptyset \neq J \subseteq \RR $ 
be an open interval, and let 
$ \Phi : J \map \RR $ 
be a differentiable function such that 
$ \varphi = \Phi ' $ is zero almost everywhere. 
Then $ \varphi (x) = 0 $ for all $ x \in J $. 
\end{lemm}

\begin{proof}
It is well-known (see \cite[Theorem 7.21]{Rud87}) 
that if a real function is differentiable 
(everywhere) and the derivative is Lebesgue integrable 
on compact intervals then 
the Newton--Leibniz formula holds for the derivative. 
Hence, for all $ x, y \in J $, $ x < y $, we have 
\[
\Phi(y) - \Phi(x) = 
\int_{x}^{y} \Phi ' (t) \,dt = 
\int_{x}^{y} \varphi (t) \,dt = 0.
\]
Of course, the integral above is the Lebesgue integral, 
and we used that $ \varphi $ is zero almost everywhere. 
Hence $ \Phi $ is constant, so 
$ \varphi = \Phi' $ is identically zero. 
\end{proof}

The main result of this section is the complete 
characterization of the solutions of \eqref{eq-maineq} 
when $ \varphi $ is a derivative. 

\begin{thm}\label{thm-main_derivative}
Let 
$ \left( \varphi \,, \psi_1 \,, \psi_2 \right) $ 
be a solution of \eqref{eq-maineq} 
and suppose that 
there exists a differentiable function 
$ \Phi : J \map \RR $ such that 
$ \varphi (u) = \Phi'(u) $ 
for every $ u \in J $. 
Then exactly one of the following statements is true. 
\begin{enumerate}
\item 
$ \varphi (u) = 0 $ for all $ u \in J $ or 
there exists $ \lambda \in \RR $ such that 
$ \psi_1(x) = \psi_2(y) = \lambda $ 
for all $ x \in I_1 $ and $ y \in I_2 \,$. 
\item 
There exist constants 
$ \lambda, \nu \in \RR $ and, for 
$ i = 1,2 $, there exist 
$ a_i \,, b_i \in \RRx $ such that 
$ \inf I_i \leq a_i \leq b_i \leq \sup I_i $ and 
for the open intervals 
\[
U_i := \left] \, \inf I_i \,, a_i \, \right[ 
\qquad \mbox{ and } \qquad 
V_i := \left] \, b_i \,, \sup I_i \, \right[ 
\]
the following assertions are fulfilled: 
\begin{itemize}
\item 
$ U_i \neq I_i $ and $ V_i \neq I_i $ 
for $ i = 1,2 \,$; 
\item 
$ U_1 \,, U_2 $ are both nonempty or 
$ V_1 \,, V_2 $ are both nonempty; 
\item 
$ \psi_1 (x) =  \psi_2 (y) = \lambda $ for all 
$ (x,y) \in U_1 \times U_2 $ and 
$ \psi_1 (x) =  \psi_2 (y) = \nu $ for all 
$ (x,y) \in V_1 \times V_2 \,$; 
\item 
$ \varphi (u) = 0 $ for all 
$ u \in 
\frac{1}{2} \left( K_1 + I_2 \right) \cup 
\frac{1}{2} \left( I_1 + K_2 \right) $ 
where 
$ K_i := I_i \setminus \left( U_i \cup V_i \right) $ 
is a nonempty subinterval 
which is closed in $ I_i $ ($ i=1,2 $). 
\end{itemize}

\item 
There exist a constant $ \lambda \in \RR $ 
and an index 
$ j \in \lbrace 1,2 \rbrace $ such that 
$ \psi_j(z) = \lambda $ for all $ z \in I_j \,$, 
while there exists a finite or countable collection of 
disjoint, nonempty open intervals 
$ U_n  \subset I_i $ 
($ i \neq j $ and $ n = 1, \dots, N $ 
where $ N \in \NN $ or $ N = \infty $) 
with the following properties: 
\begin{align*}
& \bigcup_{n=1}^N U_n \neq I_i 
\ \mbox{ and } \ \psi_i(z) = \lambda 
\ \mbox{ for all } \ z \in \bigcup_{n=1}^N U_n \,, 
\mbox{ moreover } \\  
& \varphi(t) = 0 \ \mbox{ for all } \ 
t \in \frac{1}{2} \left( 
\left( I_i \setminus \bigcup_{n=1}^N U_n \right) + I_j 
\right).  
\end{align*}
\end{enumerate} 
Conversely, if there exist nonempty open intervals 
$ I_1 \,, I_2 \subseteq \RR $, 
$ J = \frac{1}{2} \left( I_1 + I_2 \right) $ 
such that for the functions 
$ \psi_1 : I_1 \map \RR $, 
$ \psi_2 : I_2 \map \RR $ and 
$ \varphi : J \map \RR $ one 
of the three cases above holds, 
then the triplet 
$ \left( \varphi \,, \psi_1 \,, \psi_2 \right) $ 
is a solution of \eqref{eq-maineq}. 
\end{thm}

\begin{proof}
Firstly let us observe that the direct implication 
is a special case of Theorem \ref{thm-main_measurable}, 
because $\varphi$ is a derivative, so it is measurable. 
However, in case \emph{(1)} and \emph{(2)} we can 
improve the assertions 
\[
\mu \left( \supp \varphi \right) = 0 
\ \mbox{ and } \ 
\mu \left( \left( 
\frac{1}{2} \left( K_1 + I_2 \right) \, \cup \, 
\frac{1}{2} \left( I_1 + K_2 \right) \right)
\cap \supp  \varphi \right) = 0 
\]
to 
\[
\varphi(u) = 0 
\, \mbox{ for all }
u \in  J
\hspace{5mm} \mbox{ and } \hspace{5mm} 
\varphi(u) = 0 
\, \mbox{ for all }
u \in  
\frac{1}{2} \left( K_1 + I_2 \right) \cup 
\frac{1}{2} \left( I_1 + K_2 \right),
\]
respectively. This fact is the consequence 
of Lemma \ref{lemm-derivative_ae00}, which 
ensures that if 
$ \varphi $ is zero almost everywhere on an open interval, 
the $ \varphi $ is identically zero on that interval. 
The converse implication was verified by T.~Kiss during 
the proof of \cite[Theorem 6.]{Kis24}. 
\end{proof}


\end{document}